\documentclass{article}
\usepackage{amsmath}
\usepackage{amsthm}
\usepackage{amssymb}
\usepackage{color}
\usepackage{latexsym}

\newtheorem{theorem}{Theorem}[section]
\newtheorem{def.}[theorem]{Definition}
\newtheorem{cor.}[theorem]{Corollary}
\newtheorem{lemma}[theorem]{Lemma}
\newtheorem{definition}[theorem]{Definition}
\newtheorem{proposition}[theorem]{Proposition}
\newtheorem{observation}[theorem]{Observation}

\theoremstyle{remark}
\newtheorem*{remark}{Remark}

\def\ZZ{{\mathbb{Z}}}
\def\RR{{\mathbb{R}}}

\def\NN{{\mathbb{N}}}

\def\RRd{{\mathbb{R}^d}}

\newcommand{\card}{\hbox{\ensuremath{\#}}}

\begin{document}

\title{Frames by Multiplication%
}
\author{P. Balazs\thanks{
Acoustics Research Institute, Austrian Academy of Sciences,
Wohllebengasse 12-14, 1040 Wien, Austria. E-mail:
peter.balazs@oeaw.ac.at.} \and C. Cabrelli\thanks{Departamento de
Matem\'atica, Facultad de Ciencias Exactas y Naturales,
Universidad de Buenos Aires, Pabell\'on I, Ciudad Universitaria,
C1428EGA C.A.B.A., Argentina, and IMAS, CONICET, Argentina. E-mail:
cabrelli@dm.uba.ar, sheinek@dm.uba.ar, umolter@dm.uba.ar.\newline
For manuscript correspondence: Sigrid Heineken, Departamento de
Matem\'atica, Facultad de Ciencias Exactas y Naturales,
Universidad de Buenos Aires, Pabell\'on I, Ciudad Universitaria,
C1428EGA C.A.B.A., Argentina, tel/fax:+541145763335, E-mail:
sheinek@dm.uba.ar.}
 \and S. Heineken$^{\dag}$ \and U. Molter$^{\dag}$}


\maketitle
\begin{abstract}
In this note we study frame-related properties of a sequence of functions multiplied by another function.
In particular we study frame and Riesz basis properties.
We apply these results to sets of irregular translates of a bandlimited function $h$ in $L^2(\RR^d)$.
This is achieved by looking at a set of exponentials restricted to a set $E \subset \RR^d$ with frequencies in a countable set $\Lambda$ and multiplying it by the Fourier transform of a fixed function
$h \in L^2(E)$.
Using density results due to Beurling, we prove the existence and give ways to construct
frames by irregular translates.

\end{abstract}




\section{Introduction}

Signal processing tools and algorithms are central in the
technology of the 21$^{\rm st}$ century. These tools  and
algorithms are used in digital instruments that have become
indispensable in everyday life. There is a wide spectrum of
devices ranging from medical applications to mass consumer
gadgets, such as cameras, Smart Phones, MP3-players or high
resolution TV. They could not exist without the  recent
development of sophisticated tools and techniques.

Signal processing is an area that for over 50 years belonged almost exclusively to engineering. Recently the
digital revolution has produced a considerable increase of the need for more mathematics to tackle difficult problems,
and for the design of new  algorithms and the refinement of existing ones. This created a rich and fruitful  interaction between both fields.

One example of this is the concept of {\em frames} \cite{duffschaef1}
which has been established as an important background for sampling
theory and signal processing. Different from bases, frame
decompositions are redundant. This property is advantageous in many
applications such as de-noising \cite{majxxl10}, error robustness or sparsity \cite{grib-nielsen1}.

Frames of translates \cite{ole1} are an important class of frames that have a special structure.
 Here, one generating function $h$ is shifted to create the analyzing family of elements, $\left\{ h \left( x - k  b \right) \right\}$. This topic was investigated in \cite{Cachka01} and can be used in Gabor or wavelet theory \cite{Bentreib01}.
 On the other side these frames are very important in the theory of Shift Invariant Spaces (SIS) \cite{bo00,dedero94} that are central in Approximation, Sampling and Wavelet theory.

In signal processing, a typical example of frames of translates are linear time-invariant filters, i.e. convolution operators.
Irregular sampling appears in numerous problems in applications, for example when dealing with jittered samples in sound analysis.
Introducing irregular shifts gives rise to an interesting generalization of those frames.
Irregular frames of translates were investigated for example in \cite{aldrgroech1}.
But there are still many open questions for this case.

Frames of translates are connected to the concept of Gabor
multipliers \cite{feinow1} by the Kohn-Nirenberg correspondence.
These operators by themselves form an interesting subclass of
time-variant filters \cite{hlawatgabfilt1}. Also in this case
irregular shifts are interesting \cite{xxlphd1}, for example in
case of non-uniform frequency sampling as in a scale adapted to
human perception \cite{xxllabmask1}.
Translates of a given function become exponentials multiplied by a
fixed function through the Fourier transform.

In this note we study the general properties of a sequence of functions $\Psi = \{\psi_k\}$ multiplied by a fixed function $h$.
In particular we investigate sufficient and necessary conditions for $\Psi$ and $h$ such that $h \cdot \Psi = \{ h \cdot \psi_k \}$ have a Bessel, frame or Riesz basis property.

For an application to frames of translates we study which properties of a set of exponentials of irregular frequencies
are preserved when they are multiplied by a fixed function $h$, and
we characterize those functions $h$ that preserve these properties.
This gives  properties of the set of translates of the inverse
Fourier transform of $h$. On the other side, using density results
due to Beurling \cite{Beu66}, we prove the existence, and provide a
method to construct, frame sequences of irregular translates.

In particular, we show that for a bounded set $E$ and any integrable
function $h$ in $P_E$ (the space of functions  in $L^2(\RR^d)$ whose
Fourier transform is supported on $E$ (c.f. \eqref{eqpe})) the set
$\{h(\cdot - \lambda_k)\}_{k\in K}$ cannot be a frame for $P_E$ even
though  the exponential functions $\{e^{-2\pi i\lambda_k x}\}_{k\in
K}$ are a frame for $L^2(E)$. However, we can choose a Schwartz
class function {\em outside} $P_E$, whose translates  on a slightly
larger set $\{\lambda'_k\}_{k\in K}$, allow us to obtain
reconstruction formulae for any function in $P_E$.

The organization of the paper is as follows. In Section 2 we set the
notation, give basic definitions and state known results. In Section
3 we give conditions in order that a sequence that form a
frame (Bessel sequence or Riesz sequence), remains a frame (Bessel sequence or Riesz sequence respectively) when multiplied by an appropriate function.
Finally in section 4 we apply these results to the study of frames
of translates.

\section{Preliminaries and Notation}

Throughout the article $E$ will denote a bounded subset of $\RR^d$ and $K$ will be a countable index set.

Let $\mathcal{H}$ be a separable Hilbert space.
A sequence $\{\psi_k\}_{k\in K}\subseteq \mathcal{H}$ is a {\sl frame} for $\mathcal{H}$ if there exist positive
constants $A$ and $B$ that satisfy
$$A\|f\|^2\leq\sum_{k\in K}|\langle f,\psi_k\rangle|^2 \leq
B\|f\|^2\,\,\,\,\,\forall f \in \mathcal{H}.$$

If $A=B$ then it is called a {\sl tight frame}. If $\{\psi_k\}_{k\in K}$ satisfies the right inequality in the above
formula, it is called a {\sl Bessel sequence}.

A sequence $\{\psi_k\}_{k\in K}$ is a {\sl Riesz basis} for $\mathcal{H}$ if it is complete in
$\mathcal{H}$ and there exist positive constants $A$ and $B$ such that for every finite scalar sequence
$\{c_k\}$ one has
$$A \sum |c_k|^2 \leq \|\sum c_k \psi_k\|^2\leq
B\sum |c_k|^2.$$

We say that $\{\psi_k\}_{k\in K}$ is a {\sl frame sequence} if it is a frame for the space it spans, and it is a {\sl Riesz sequence} if it is a Riesz basis for the space it spans.

For a closed subspace $\mathcal{V} \subseteq \mathcal{H}$ denote the
projection on it by $\mathcal{P}_{\mathcal{V}}$. A sequence
$\{\phi_k\}_{k\in K}\subseteq \mathcal{H}$ is an {\sl outer frame}
\cite{ACM04} for a closed subspace $\mathcal{V} \subseteq
\mathcal{H}$ if $\{\mathcal{P}_{\mathcal{V}} (\phi_k)\}_{k\in K}$ is
a frame for $\mathcal{V}$.

For two sets $F \subseteq G \subseteq \RRd$ we use the notation
\begin{equation}\label{sec:embedltwo1}
\widetilde{ L^2 (F)}^{(G)} := \left\{ f \in L^2(G) : f(x) = 0 \text { for } a.e. \ x \in G \backslash F \right\} .
\end{equation}
This set is isomorphic to $L^2(F)$ using $\varphi  : \widetilde{ L^2 (F)}^{(G)} \rightarrow L^2 (F)$ where $\varphi\left( f \right) = f{|_F}$.
\\
When $G = \RR^d$ we will just write $L^2(F)$ in place of $\widetilde{ L^2 (F)}^{(\RR^d)}$.

Frames of exponentials have been studied in \cite{duffschaef1}.
Conditions on a discrete set $\Lambda$ such that $\{e^{-2\pi
i\lambda x}\}_{\lambda \in \Lambda}$ is a frame or a Riesz basis for $L^2(E),$ where
$E\subseteq \RR$ is a bounded interval, have been given in \cite{pav},
\cite{Jaf}, \cite{Seip1}, \cite{Lan67}. For the case that $E$ is a finite
union of certain intervals it is known that such sets $\Lambda$
exist \cite{lyusei}. In higher dimensions, there exists results for particular sets
$E$ \cite{lyuras}, \cite{sei}.

For $\lambda \in \RR^d$, we denote by $e_{\lambda}$ the function
defined by $e_{\lambda}(x)=e^{-2\pi i\lambda x}$ and by
$T_{\lambda}$ the operator $T_{\lambda}f(x)=f(x-\lambda).$ We will
use $|E|$ to denote the Lebesgue measure of a measurable set $E.$
For standard results on integration theory we use in this article we
refer e.g. to \cite{fo99}, \cite{embe02}, \cite{gra08}. We write
$\hat {f}$ for the Fourier transform given by
$f(\omega)=\int_{\RR^d}f(x)e^{-2\pi i \omega x},dx$ for $f\in
L^1(\RR^d)\cap L^2(\RR^d),$ with the natural extension to
$L^2(\RR^d).$

Let $\Lambda=\{\lambda_k\}_{k\in K}$ be a sequence in $\RR^d$.
Throughout the paper when we say that a set of exponentials $\{e_{\lambda_k}:{k\in K} \}$ is a frame (or  a Riesz basis) of $L^2(E)$  we will mean that the set $\{e_{\lambda_k} \chi_E\}_{k\in K}$ has that property. Here  $\chi_E$ stands for the indicator function of $E$.

\subsection{Existence of irregular exponentials frames}

In \cite{Beu66} Beurling gave sufficient conditions on a discrete set $\Lambda$ in $\RR^d$, in order that the associated
exponentials $\{e_{\lambda}\}_{\lambda \in \Lambda}$
form a frame when restricted to a ball. These conditions are given in terms of density.
In this section we review those results that will be used later.


\begin {definition}
A set $\Lambda$ is separated if
$$
\inf_{\lambda \ne \lambda'}\|\lambda-\lambda' \|>0.
$$
\end {definition}
There are many notions for the density of a set $\Lambda$. We start with
definitions that
are due to Beurling.
\begin {definition}
\
\begin {enumerate}
\item A lower uniform density $D^-(\Lambda)$ of a separated set $\Lambda\subset \RR^d$ is defined as
$$
D^-(\Lambda)=\lim\limits_{r\to \infty} \frac {\nu^-(r)} {(2r)^d}
$$
where $\nu^-(r):=\min\limits_{y \in \RR^d} \card\left({\Lambda\cap
(y+[-r,r]^d)}\right)$, where $\card(Z)$ denotes the cardinal of the
set $Z$.
\item An upper uniform density $D^+(\Lambda)$ of a separated set $\Lambda$ is
defined as
$$
D^+(\Lambda)=\lim\limits_{r\to \infty} \frac {\nu^+(r)} {(2r)^d}
$$
where $\nu^+(r):=\max\limits_{y \in \RR} \#\left( {\Lambda\cap
(y+[-r,r]^d)}\right)$.
\item If $D^-(\Lambda)=D^+(\Lambda)=D(\Lambda)$, then $\Lambda$ is said to have uniform
Beurling density $D(\Lambda)$.
\end {enumerate}
\end {definition}
\begin {remark}

The existence of the limits in the definitions of $D^-(\Lambda)$
and $D^+(\Lambda)$ is a consequence of the separateness of $\Lambda.$

\end {remark}
As an example, let $l > 0$ and  $\Lambda=\{\lambda_j : j\in\ZZ
\}\subset \RR$ be separated sequence such that there exists $L>0$
with  $ |\lambda_j-\frac {j} {l}| \le L$, for all $j \in \ZZ$. Then
$D^-(\Lambda)=D^+(\Lambda)=l$.

For the one dimensional case, Beurling proved
the following theorem.
\begin {theorem} (Beurling)
\label {B1}
Let $\Lambda \subset \RR$ be separated, $a  > 0$ and $\Omega=[-\frac{a} {2},\frac{a}
{2}]$. If $a<D^-(\Lambda)$  then
$\{e_{\lambda}\chi_{\Omega}\}_{\lambda \in \Lambda}$ is a frame for $ L^2 (\Omega)$.
\end {theorem}
The previous result however is only valid in one dimension. For higher
dimensions, Beurling introduced the following notion:
\begin {definition}
The gap $\rho$ of the set $\Lambda$ is defined as
$$
\rho=\rho(\Lambda)=\inf \left\{ {r>0: \, \bigcup_{\lambda \in \Lambda} B_r(\lambda)=\RR^d}\right\}
$$
\end {definition}
Equivalently, the gap $\rho$ can be defined as
$$\rho=\rho(\Lambda)=\sup_{x\in \RR^d} {\inf_{\lambda\in
\Lambda}|x-\lambda|}.
$$

It is not difficult to show that if $\Lambda$ has gap $\rho$, then
$D^-(\Lambda)\ge \frac {1}{2\rho}$.
For a separated set $\Lambda$, and for the case where
$\Omega$ is the ball $B_{r}(0)$ of radius $r$ centered at the origin,
Beurling proved the following result:
\begin {theorem}[Beurling]
\label {beurlingthm} Let $\Lambda\subset \RR^d$ be separated, and
$\Omega=B_{r}(0)$. If $r\rho <\frac{1}{4}$, then
$\{e_{\lambda}\chi_{\Omega}\}_{\lambda \in \Lambda}$ is a frame for
$ L^2 (\Omega)$.
\end {theorem}

Note that actually the same set of exponentials is also a frame for $L^2(B_{r}(x))$ for any vector $x\in \RR^d$.

Using these results, in order to construct a frame of exponentials
of $L^2(E)$ for a bounded set $E \subset \RR^d$, it is enough to
find a separated sequence $\Lambda$ in $\RR^d$ with gap $\rho <
\frac{1}{4r}$ with $r $ the radius of a ball containing $E$.

\section{Frames by Multiplication} \label{sec:framesmultipl}

In this section we will investigate the properties of the sequence $\{\varphi \cdot \psi_k\}$ given the sequence $\{\psi_k\}$ and the function $\varphi$.

\begin{lemma}\label{compl}
Let $\{\psi_k\}_{k \in K}$ be complete in $L^2(E).$ Let
$\varphi\in L^2(\RR^d)$ such that $\{\varphi\psi_k\}$ is in $L^2(E)$ and $|\{t\in E :\varphi(t)=0\}|=0.$ Then
$\{\varphi\psi_k\}$ is complete in $L^2(E).$
\end{lemma}
\begin{proof}
Assume $f\in L^2(E)$ and $<f, \varphi \psi_k>=0$ for every $k
\in K.$ Then
\begin{equation} \label{zero-k}
\int_E f \overline{\varphi \psi_k}= \int_E (f
\overline{\varphi})\overline{\psi_k}=0 \text{ for every } k
\in K. \end{equation}

Let $L$ be large enough so that $E\subseteq [-\frac{L}{2},\frac{L}{2}]^d$,
$\varepsilon > 0$ and set $g = f \overline{\varphi}$.

Note that since $g\in L^1(E),$ there exists $\delta
>0$ such that $\int_A |g|< \varepsilon,$ for every set $A$ such that
$|A|<\delta.$

Let now $n$ be in $\ZZ^d.$  Since $\{\psi_k\}_{k \in K}$ is
complete in $L^2(E),$ we have a sequence  $\{f_m\}_{m \in \NN}$ in
${\rm span}\{\psi_k\}_{k \in K}$ that converges  to
$e_{\frac{n}{L}}\chi_E$ in $L^2(E)$. So there exists a subsequence
$\{f_{m_l}\}_{l\in \NN}$ that converges $a.e.$ to $e_{\frac{n}{L}}\chi_E.$

Since $E$ has finite measure, by  Egorov's theorem  we can choose a closed subset
$F$ of $E$ such that $|E\setminus F|<\delta$ and $\{f_{m_l}\}_{l\in
\NN}$ converges uniformly to $e_{\frac{n}{L}}\chi_E$ on $F.$

So we have
$$\left|\int_E ge_{\frac{n}{L}}\right| \leq \left|\int_F ge_{\frac{n}{L}}\right| +\left| \int_{E\setminus F}
ge_{\frac{n}{L}}\right|
\leq \lim_l \left| \int_F g f_{m_l} \right| + \left| \int_{E\setminus F} g e_{\frac{n}{L}}\right|
\leq \varepsilon
$$

Since $\varepsilon$ is arbitrary, it follows that$\int_E
ge_{\frac{n}{L}}=0$  for every integer $n.$

Let $\tilde g$ be the extension of $g$ to
$[-\frac{L}{2},\frac{L}{2}]^d$, which is zero a.e in
$[-\frac{L}{2},\frac{L}{2}]^d \setminus E$. Note that $\tilde g$ is
in $L^1([-\frac{L}{2},\frac{L}{2}]^d).$ Now, using the completeness
of $\{e_{\frac{n}{L}}\}_{n \in \ZZ^d}$, in
$L^2([-\frac{L}{2},\frac{L}{2}]^d)$ applying a similar argument as
in the proof of Theorem 2 in \cite{young}, we obtain that
$\tilde{g}=0$ a.e. in $[-\frac{L}{2},\frac{L}{2}]^d.$ Since
$\varphi\neq 0$ a.e in $E,$ it follows that $f=0$ a.e. in $E.$
\end{proof}

\begin{proposition}\label{fr}
Let $\varphi$ be a measurable function defined in $\RR^d$ and $\{\psi_k\}_{k \in K}$ be a frame of $L^2(E).$ Then
\begin{enumerate}
\item $\{\varphi\psi_k\}_{k \in K}$ is a frame of
$L^2(E)$ if and only if there exist constants $A$ and $B$ such
that
\begin{equation}\label{const}
0<A \leq B<+ \infty \,\,\,\,\text{ and }\,\, A\leq
|\varphi(t)|\leq B \,\,\,\,\,\text{ a.e. } t\in E.
\end{equation}

\item If $\varphi\in L^2(\RR^d)$ such that   $\{\varphi\psi_k\}$ is in $L^2(E)$ and $|\{t\in E:\varphi(t)=0\}|=0,$ then
$\{\varphi\psi_k\}$ is complete in $L^2(E).$
\end{enumerate}

\end{proposition}
\begin{proof}

Part 1.:

$\Longrightarrow$)

Assume that both $\{\psi_k\}_{k \in K}$ and
$\{\varphi \psi_k\}_{k \in K}$  are frames of
$L^2(E).$

Assume that for every $A>0,$ there exists a set $U \subseteq E$ of
positive measure such that $|\varphi(t)|< A$ for every $t\in U.$

For $n\in\NN$, let $E_n=\{t\in E:|\varphi(t)|<\frac{1}{n}\}.$
Note that $|E_n|>0$ for every $n\in\NN.$  Define
\begin{equation}\label{sec:fnnn1}
f_n(t)=
\begin{cases}
\frac{1}{\sqrt{|E_n|}}& \text{for } t \in E_n \\
 0 & \text{otherwise}.
\end{cases}
\end{equation}

We have that $\|f_n\|_2=1$ for every $n\in\NN$ and so $f_n \in
L^2(E)$.

If $\alpha$ is the lower frame bound of
$\{\varphi\psi_k\}_{k \in K}$ and $M$ is the upper
frame bound of $\{\psi_k\}_{k \in K},$ then

\begin{align*}
 \alpha & \leq\sum_{k \in K}|<f_n,\varphi\psi_k>|^2
\\
& =\sum_{k \in K}|<f_n\overline{\varphi},\psi_k>|^2
\leq M \|f_n\varphi\|_2^2\\
& =M\int_{E_n}|f_n \varphi|^2 =
\frac{1}{|E_n|}M \int_{E_n}|\varphi|^2 \leq
\frac{M}{n^2}\longrightarrow 0,
\end{align*}
which is a contradiction. So we can conclude that there exists an
$A>0$ such that $A\leq|\varphi(t)|$ a.e $t \in E.$

To prove the existence of the upper bound in \eqref{const}, assume
that  for every $B>0$ there exists a set $V \subseteq E$ of positive measure,
such that $|\varphi(t)|> B$ for every $t\in V.$

For $s\in\NN$, let $E_s=\{t\in E:|\varphi(t)|> s\}.$ We have
that $|E_s|>0$ for every $s\in\NN.$  Define
$f_s(t)$ like in \eqref{sec:fnnn1}.

Let $m$ be the lower frame bound of
$\{\psi_k\}_{k \in K}.$ Then

\begin{align*}
\sum_{k \in K}|<f_s,\varphi\psi_k>|^2 & = \\
& =\sum_{k \in K}|<f_s\overline{\varphi},\psi_k>|^2
\geq m \|f_s\varphi\|_2^2\\
& = m\int_{E_s}|f_s \varphi|^2 =
\frac{1}{|E_s|}m \int_{E_s}|\varphi|^2 \geq m s^2\longrightarrow
+\infty,
\end{align*}

which again is a contradiction, so there must exist a constant $B$
that satisfies \eqref{const}.

$\Longleftarrow$)

Assume there exist positive constants $A,B>0$ such that $A\leq
|\varphi(t)|\leq B \,\,\,\text{a.e. } t \in E$ and that
$\{\psi_k\}_{k \in K}$ is a frame of $L^2(E)$ with
lower and upper frame bounds $m$ and $M$ respectively.

Since $\varphi \in L^\infty(E)$, $f \varphi \in L^2(E)$ for $f\in
L^2(E).$

$$\sum_{k \in K}|<f,\varphi\psi_k>|^2=
\sum_{k \in K}|<f\overline{\varphi},\psi_k>|^2,$$

and so we have that $$m \|f\overline{\varphi}\|^2\leq \sum_{k \in
K}|<f,\varphi\psi_k>|^2\leq M \|f\overline{\varphi}\|^2
\,\,\,\,\text{ for every } f \in L^2(E).$$

But

$$\|f\overline{\varphi}\|^2\geq
A^2\|f\|^2\,\,\,\,\,\,\text{and}\,\,\,\,\,\,
\|f\overline{\varphi}\|^2\leq B^2\|f\|^2,$$

so

$$m A^2\|f\|^2 \leq
\sum_{k \in K}|<f,\varphi\psi_k>|^2\leq M B^2\|f\|^2
\,\,\,\,\text{ for every } f \in L^2(E). $$

This completes the proof of part 1.

Part 2. is an immediate consequence of Lemma \ref{compl}.

\end{proof}


Analogously, the following result can be proved:

\begin{proposition}\label{tfr}
Let $\varphi$ be a measurable function defined in $\RR^d.$ Let $\{\psi_k\}_{k \in K}$ be a tight frame of $L^2(E).$ Then
$\{\varphi\psi_k\}_{k \in K}$ is a tight frame of
$L^2(E)$ if and only if there exists a positive constant $A$ such
that
\begin{equation}\label{tconst}
|\varphi(t)|=A \,\,\,\,\,\text{ a.e. } t\in E.
\end{equation}
\end{proposition}

On the same lines, we can also obtain a similar result for Riesz bases instead of frames:

\begin{proposition} \label{riesz}
Let  $\varphi$ be a measurable function defined in $\RR^d.$ Let
$\{\psi_k\}_{k \in K}$ be a Riesz basis of $L^2(E).$ Then
$\{\varphi\psi_k\}_{k \in K}$ is a Riesz basis of
$L^2(E)$ if and only if there exist constants $A$ and $B$ such
that
\begin{equation}\label{constseq}
0<A \leq B<+ \infty \,\,\,\,\text{ and }\,\, A\leq
|\varphi(t)|\leq B \,\,\,\,\,\text{ a.e. } t\in E.
\end{equation}
\end{proposition}

\begin{proof}

$\Longrightarrow$)

Since $\{\varphi\psi_k\}_{k \in K}$ is a Riesz basis of
$L^2(E)$ it is a frame of $L^2(E),$ so by Proposition~\ref{fr} there exist constants $A$ and $B$ such
that inequality (\ref{constseq}) holds.

$\Longleftarrow$)

If there exist constants $A$ and $B$ such
that inequality (\ref{constseq}) holds, then by Proposition~\ref{fr}
$\{\varphi\psi_k\}_{k \in K}$ is a frame of $L^2(E).$ Hence, for every
$f\in L^2(E),$
\begin{equation}\label{des}
f=\sum_{k \in K} c_{\lambda_k}\varphi\psi_k,
\end{equation}
where $\{c_{\lambda_k}\}_{\lambda_k \in \Lambda}\in \ell^2(\Lambda).$
To see that the coefficients $\{c_{\lambda_k}\}_{\lambda_k \in \Lambda}$ in (\ref{des}) are unique, we observe that since
$|\varphi(t)|\geq A \,\text{ a.e. } t\in E,$
\begin{equation}
\frac{{f}}{\varphi} = \sum_{k \in K}
c_{\lambda_k}\psi_k\,\,\,\in L^2(E).
\end{equation}
The result follows using that $\{\psi_k\}_{k \in K}$ is a Riesz basis of $L^2(E).$
\end{proof}

\begin{proposition} \label{bessel}

Let  $\varphi$ be a measurable function defined in $\RR^d.$ Let $\{\psi_k\}_{k \in K}$ be a frame of $L^2(E).$ Then
$\{\varphi\psi_k\}_{k \in K}$ is a  Bessel sequence of
$L^2(E)$ if and only if there exists a constant $B > 0$ such
that
\begin{equation}\label{constb}
|\varphi(t)|\leq B \,\,\,\,\,\text{ a.e. } t\in E.
\end{equation}

\end{proposition}

\begin{proof}
One can apply the same arguments as in proof of Proposition~\ref{fr} part 1.
\end{proof}

\begin{proposition} \label{conv}
Let  $\varphi$ be a measurable function defined in $\RR^d.$ If
there exist constants $A$ and $B$ such
that
\begin{equation}\label{bound}
0<A \leq B<+ \infty \,\,\,\,\text{ and }\,\, A\leq
|\varphi(t)|\leq B \,\,\,\,\,\text{ a.e. } t\in E,
\end{equation}

and $\{\varphi\psi_k\}_{k \in K}$ is a frame of $L^2(E),$
then $\{\psi_k\}_{k \in K}$ is a frame of $L^2(E).$
\end{proposition}

\begin{proof}

Let $\alpha$ and $\beta$ be the lower respectively the upper frame bound of $\{\varphi\psi_k\}_{k \in K}.$
Since  $\{\varphi\psi_k\}_{k \in K}$ is a frame, for every $f\in L^2(E)$ we can write

\begin{align*}
\sum_{k \in K}|<f,\psi_k>|^2 & =\sum_{k \in K}|<f,\frac{1}{\varphi} \varphi \psi_k>|^2=
\sum_{k \in K}|<f\frac{1}{\overline{\varphi}},\varphi \psi_k>|^2\\
& \leq
\beta \|f\frac{1}{\overline{\varphi}}\|^2\leq \frac{\beta}{A^2} \|f\|^2.
\end{align*}

Analogously we obtain

$$\sum_{k \in K}|<f,\psi_k>|^2\geq \frac{\alpha}{B^2}\|f\|^2.$$
\end{proof}

\begin{observation}
\begin{itemize}
\item Clearly Proposition~\ref{conv} remains true if we replace "frame" by "Riesz basis".
\item If we replace
condition~(\ref{bound}) by the condition that there exists a positive constant $A$ such that
\begin{equation}
A\leq |\varphi(t)|  \,\,\,\,\,\text{ a.e. } t\in E
\end{equation}
then Proposition~\ref{conv} is also true when we replace "frame" by "Bessel sequence".
\end{itemize}
\end{observation}

The results can also be extended to frame sequences:

\begin{proposition}\label{framseq}
 Let $\{\psi_k\}_{k\in K}$ be a frame of $L^2(E)$ and $\varphi\in L^2(\RR^d)$ such that $\{\varphi\psi_k\}$ is in $L^2(E).$ Let $F := {\rm supp} (\varphi) \cap E$.
Then
\begin{enumerate}
\item 
$\overline{\rm span} \left\{ \varphi \psi_k \right\} = \widetilde{ L^2 (F)}^{(E)}$.

\item $\{\varphi\psi_k\}_{k\in K}$ is a frame sequence of
$L^2(E)$ if and only if there exist constants $A$ and $B$ such
that
\begin{equation}\label{constframseq}
0<A \leq B<+ \infty \,\,\,\,\text{ and }\,\, A\leq
|\varphi(t)|\leq B \,\,\,\,\,\text{ a.e. } t\in F.
\end{equation}

\item Let $\varphi$ be compactly supported. Then $\{\varphi\psi_k\}_{k\in K}$ is a frame sequence of
$L^2(\RRd)$ if and only if there exist constants $A$ and $B$ such
that
\begin{equation}
0<A \leq B<+ \infty \,\,\,\,\text{ and }\,\, A\leq
|\varphi(t)|\leq B \,\,\,\,\,\text{ a.e. } t\in F.
\end{equation}

\end{enumerate}
\end{proposition}
\begin{proof}
\

\begin{enumerate}
\item
Clearly for each $k\in K$ we have  $\varphi \psi_k \in \widetilde{ L^2 (F)}^{(E)}$, and so
$$V := \overline{\rm span} \left\{ \varphi \psi_k :k\in K \right\}
\subseteq \widetilde{ L^2 (F)}^{(E)} $$
as this is a closed subspace.

On the other hand due to Lemma \ref{compl} $\overline{\rm span}\left\{ \varphi \psi_k: k \in K \right\} = L^2(F) \cong \widetilde{ L^2 (F)}^{(E)}$. Therefore $V = \widetilde{ L^2 (F)}^{(E)}$.

\item

Using (1.) the second part is equivalent to
$\{\varphi\psi_k\}_{k\in K}$ is a frame for
$\widetilde{ L^2 (F)}^{(E)} \cong L^2 (F)$ if and only if there exist constants $A$ and $B$ such
that
$$ 0<A \leq B<+ \infty \,\,\,\,\text{ and }\,\, A\leq
|\varphi(t)|\leq B \,\,\,\,\,\text{ a.e. } t\in F.
$$

This is just Proposition \ref{fr} applied to  $L^2 (F)$.

\item

As $F$ is bounded, just choose a bounded set $E\supset  F$ and apply part 2. Note that
for this $E \subseteq \RRd$ we have  $\widetilde{ L^2 (F)}^{(E)} \cong  \widetilde{ L^2 (F)}^{(\RRd)}$.
\end{enumerate}
\end{proof}

\section{Application to Frames of Translates} \label{sec:appframtransl0}

Now we apply the results proved above for the special case of exponential functions and the Fourier transform of a generator to arrive at results for frames of translates.

We denote
\begin{equation}
P_E=\left\{ f \in L^2(\RRd) : {\rm supp \ } \hat f \subseteq E
\right\}.\label{eqpe}
\end{equation}
\begin{theorem}\label{pe}
Let $\Lambda=\{\lambda_k\}_{k\in K}\subseteq\RR^d$ such that
$\{e_{\lambda_k}\}_{k\in K}$ is a frame for $L^2(E).$ Let $h \in
P_E.$ Then
\begin{enumerate}
\item $\{T_{\lambda_k} h\}_{k\in K}$ is a Bessel sequence in $L^2
\left(\RRd\right)$ if and only if there exists $B > 0$ such that
$\left| \hat h (\omega) \right| \le B$ a.e.

\item $\{T_{\lambda_k} h\}_{k\in K}$ is a frame for $P_E$ if and only if there
exist $B \ge A > 0$ such that  $A \le \left| \hat h (\omega) \right|
\le B$ for a.e. $\omega \in E$.

\item $\{T_{\lambda_k} h\}_{k\in K}$ is a frame sequence in $L^2
\left(\RRd\right)$ if and only if there exist $B \ge A > 0$ such
that  $A \le \left| \hat h (\omega) \right| \le B$ for a.e. $\omega
\in {\rm supp \ } \hat h$.
\end{enumerate}

\end{theorem}

\begin{proof}
\

\begin{enumerate}
\item Let $f \in L^2(\RRd),$
\begin{equation} \label{sec:bessframtrans2}  \sum \limits_{k\in K} \left| \left< f , T_{\lambda_k} h \right>_{L^2 \left(\RRd\right)} \right|^2  =
\sum \limits_{k\in K}
\left| \left< \hat f , e_{\lambda_k} \hat h \right>_{L^2(E)} \right|^2.
\end{equation}

Assume $\{T_{\lambda_k} h\}_{k\in K}$ is a Bessel sequence in $L^2
\left(\RRd\right).$ Then there is a $\beta>0$ such that $\eqref{sec:bessframtrans2}\leq
\beta \left\| \hat f \right\|_{L^2(\RRd)}$ for every $f\in
L^2(\RRd).$ For $g\in L^2(E),$ we know that there exists an $f\in
P_E$ such that $g=\hat{f}_{|_E}\,\,a.e.$ Hence we have that $$\sum
\limits_{k\in K} \left| \left< g , e_{\lambda_k} \hat h
\right>_{L^2(E)} \right|^2 \leq \beta \left\| \hat{f}
\right\|_{L^2(\RRd)}=\beta \left\| g \right\|_{L^2(E)}
\,\,\,\text{for every}\,\,\, g\in L^2(E).$$ So, by
Proposition~\ref{bessel} there exists $B > 0$ such that $\left| \hat
h (\omega) \right| \le B$ for a.e $\omega \in E $. As ${ \rm supp \
} \hat h \subseteq E$ this implies that $\left| \hat h (\omega)
\right| \le B$ a.e.

For the other implication, assume $\left| \hat h (\omega) \right| \le B$ a.e.
By Proposition~\ref{bessel} we have that \eqref{sec:bessframtrans2} $\leq B \left\| \hat{f}
\right\|_{L^2(E)}\leq B \left\| \hat f
\right\|_{L^2(\RRd)}=B \left\| f\right\|_{L^2(\RRd)}$ for every $f\in L^2(\RRd),$ i.e.
$\{T_{\lambda_k} h\}_{k\in K}$ is a Bessel sequence in $L^2
\left(\RRd\right).$

\item \& 3.

For $f \in P_E$ we have
$$  \sum \limits_{k\in K} \left| \left< f , T_{\lambda_k} h \right>_{P_E} \right|^2  =
\sum \limits_{k\in K} \left| \left< \hat f , e_{\lambda_k} \hat h
\right>_{L^2(E)} \right|^2 . $$ Note that $\left\| f \right\|_{P_E}
= \left\| \hat f \right\|_{ L^2 (E)}$. So the statements are a
direct consequence of Proposition~\ref{bessel} and
Proposition~\ref{framseq}.

\end{enumerate}

\end{proof}

This implies the following interesting Corollary.

\begin{cor.}\label{cor} Let $h\in P_E$ such that $\hat h$ is continuous. Then there does not exist $\Lambda=\{\lambda_k\}_{k\in K} \subseteq \RRd$
such that $\{h(\cdot-\lambda_k)\}_{k\in K}$ is a frame of
$P_E.$
\end{cor.}

\bigskip

\noindent{\bf Example:}

\noindent Consider the Paley Wiener space

$$P_{1/2}=\left\{ f \in L^2(\RR) : {\rm supp \ } \hat f \subseteq
\left[-\frac{1}{2},\frac{1}{2}\right] \right\},$$

which is generated by $\psi(x)=\frac{\sin\pi x}{\pi x}.$

The translates $\{\psi(\cdot-k)\}_{k\in\ZZ}$ are an orthonormal
basis, in particular a frame for $P_{1/2}.$

If $h\in P_{1/2} \cap L^1(\RR)$ there does not exist
$\Lambda=\{\lambda_k\}_{k\in K}\subseteq\RR$ such that $\{h(\cdot-\lambda_k)\}_{k\in K}$ is a frame for $P_{\frac{1}{2}}.$

The result of  Corollary \ref{cor}  represents an obstacle  for applications, since any generator $h$ in
our construction will have poor decay and in consequence it will produce a big error if we need to truncate the
expansions in terms of  the translates of $h$.

However, this problem can be overcome for open bounded sets $E$ if we generalize a trick that appears in \cite{DD03} and \cite{ACM04} to this case. The price to pay for this,  is a little bit of oversampling.

\begin{theorem} \label{sec:gooddecayexpans1}

 Let $E\subset \RRd$ be an open bounded set and $\Lambda=\{\lambda_k\}_{k\in K}$ a separated sequence in $\RR^d$ such that the
 exponentials $\{e_{\lambda_k}\}_{k\in K}$ form a frame of $L^2(E)$.  Then, there exists a separated sequence $\Lambda'=\{\lambda'_k\}$ containing $\Lambda$ and a function $g$  of  the Schwartz class,  compactly supported in frequency, such that each function $f$ in $P_E$ has an expansion as
 \begin{equation}\label{expansion}
 f(x) = \sum_{k\in K} \alpha_{k} g(x-\lambda'_k),
 \end{equation}
 where the sequence of coefficients $\{\alpha_k\}_{k\in K}$ is in $l_2(\Lambda')$ and the series converges uniformly, unconditionally and in $L^2(\RRd)$.
 \end{theorem}
\begin{proof}
 For a small $\delta > 0$
let $E_{\delta} =\{x\in \RRd: d(x,E) < \delta\}$.

Consider a function $g \in L^2(\RRd)$ such that its Fourier transform $\hat g$ satisfies:
$\hat g$ is of class $C^\infty$ ,  $ 0 \le \hat g \le 1$,
$\hat g (\omega) = 1$ for $\omega \in E$, and $\hat g = 0$ for $\omega  \in \RRd \setminus E_{\delta}.$ Then $\hat g$ is a Schwartz class function, and therefore $g$ is, too.

Now we chose another separated sequence $\Lambda'$ that on one side
contains  $\Lambda$, but on the other hand satisfies that the
associated exponentials form a frame of $L^2(E_{\delta})$. This can
be done by simply adding sufficient points to $\Lambda$, to decrease
its gap.


For $f \in P_E\subset P_{E_{\delta}}$,  using the frame expansion, we have that

$$ \hat f (\omega) = \sum_k \alpha_k \;e_{\lambda'_k}(\omega) \chi_{E_{\delta}}(\omega),
\quad a.e. \;\omega \in \RRd,$$ with unconditional convergence. Note
that since supp$(\hat f) \subset E$, this implies in particular,
that $\sum_k \alpha_k \;e_{\lambda'_k}(\omega)
\chi_{E_{\delta}}(\omega) = 0$ for $\omega \in E_{\delta}\setminus
E$.

Therefore, because of the choice of the properties of $g$ we can
write

\begin{equation}\label{ja}
 \hat f (\omega) = \sum_k \alpha_k \;e_{\lambda'_k}(\omega) \hat g(\omega),
\quad a.e. \;\omega \in \RRd.
\end{equation}

Taking the inverse Fourier transform in (\ref{ja}) we have

$$
  f (x) = \sum_k \alpha_k  \;g(x-\lambda'_k),
\quad a.e. \;x \in \RRd,
$$
where the  convergence is unconditional in $L^2(\RR).$

Moreover, by Cauchy Schwartz,
$$ \left| \sum_{|k| \leq N} \alpha_k  \;g(x-\lambda'_k)\ \right|^2 \leq \|\alpha\|_{2}^{2}
\sum_{|k| \leq N}  \;|g(x-\lambda'_k)|^2.$$
The uniform convergence is therefore a straightforward consequence of
the decay of $g$ and the fact that the sequence $\Lambda'$ is separated.

\end{proof}

Notice that the function $g$ is not in $P_E$ and its translates by elements in $\{\lambda'_k\}$
do not form a frame sequence. However its orthogonal projections on $P_E$ is a frame of $P_E$. 
\begin{proposition} With the assumptions of Theorem \ref{sec:gooddecayexpans1},
$\{T_{\lambda'_k}g\}$ is an {\em outer frame} for $P_E$.
\end{proposition}
\begin{proof} Let $\mathcal{P}$ denote the orthogonal projector onto $P_E.$
For  $f\in P_E$ we have,
$$<\mathcal{P} \left(T_{\lambda'_k}g \right),f> = <T_{\lambda'_k}g, \mathcal{P} f> = <T_{\lambda'_k}g,f> = $$
$$= <e_{\lambda'_k}\hat{g},\hat{f}>
=\int_{\RR^d} (\hat{g} \overline{\hat{f}})  e_{\lambda'_k} =\overline{ \int_E \hat{f} \;\overline{e_{-\lambda'_k}}}$$

Since supp$(\hat f) \subset E$ and $\{e_{-\lambda'_k}\}_k$ forms a frame of $L^2(E)$ we have that
$\{\mathcal{P} ( T_{\lambda'_k}g )\}_{k\in K}$ forms a frame for $P_E$.
\end{proof}

The next result is about properties that are preserved under the action of convolution.
Convolution is of particular interest in applications for linear time-invariant filters.
\begin{proposition}
Let $\Lambda=\{\lambda_k\}_{k\in K}\subseteq\RR^d$ such that
$\{e_{\lambda_k}\}_{k \in K}$ is a frame for $L^2(E).$ Let
$f, g \in P_E.$ Then

\begin{enumerate}

\item \label{b1} If $\{T_{\lambda_k} f\}_{k \in K}$ is a Bessel sequence in $L^2\left(\RRd\right),$  and $\{T_{\lambda_k} g\}_{k \in K}$
is a Bessel sequence in $L^2\left(\RRd\right),$ then $\{T_{\lambda_k} (f
\ast g)\}_{k \in K}$ is a Bessel sequence in
$L^2\left(\RRd\right).$

\item \label{f1}If $\{T_{\lambda_k} f\}_{k \in K}$ is a frame for $P_E$ and $\{T_{\lambda_k} g\}_{k \in K}$
is a frame for $P_E,$ then $\{T_{\lambda_k} (f \ast
g)\}_{k \in K}$  is a frame for $P_E.$

\item \label{fs1} If $\{T_{\lambda_k} f\}_{k \in K}$ is a frame sequence in $L^2\left(\RRd\right),$  and $\{T_{\lambda_k} g\}_{k \in K}$ is a frame sequence in
$L^2\left(\RRd\right),$  then $\{T_{\lambda_k} (f \ast
g)\}_{k \in K}$ is a frame sequence in
$L^2\left(\RRd\right).$

\item \label{f2} If $\{T_{\lambda_k} f\}_{k \in K}$ is a frame for
$P_E$ and $\{T_{\lambda_k} (f \ast g)\}_{k \in K}$ is a frame
for $P_E,$ then $\{T_{\lambda_k} g\}_{k \in K}$ is a frame for
$P_E$.

\item \label{fs2}If $\{T_{\lambda_k} f\}_{k \in K}$ is a frame sequence in $L^2\left(\RRd\right)$
 and $\{T_{\lambda_k} (f \ast g)\}_{k \in K}$ is a frame sequence in $L^2\left(\RRd\right),$ then
$\{T_{\lambda_k} g\}_{k \in K}$ is a frame sequence in
$L^2\left(\RRd\right).$

\item\label{b2} Let $\{T_{\lambda_k} (f \ast g)\}_{k \in K}$ be a Bessel sequence in
 $L^2\left(\RRd\right).$ If there exists $C > 0$ such that
$\left| \hat f (\omega) \right| \geq C$ a.e., then  $\{T_{\lambda_k}
g\}_{k \in K}$ is a Bessel sequence in
$L^2\left(\RRd\right).$
\end{enumerate}
\end{proposition}

\begin{proof}

First observe that $f\ast g \in P_E.$

\ref{b1}. By Theorem~\ref{pe} there exist $B_1, B_2
> 0$ such that $\left| \hat f (\omega) \right| \le B_1 $ a.e. and
$\left| \hat g (\omega) \right| \le B_2$ a.e. Since $\widehat{f\ast
g}= \hat f \hat g,$ we have that $\left| \widehat{f\ast g} (\omega)
\right| \le B_1 B_2 $ a.e. and the result follows.

\medskip

Part \ref{f1} and part \ref{fs1} can be proved analogously.

\medskip

\ref{f2}. By Theorem~\ref{pe},

$A_1 \le \left| \hat f (\omega) \right| \le B_1 $ for almost all
$\omega \in E$ and $A_2 \le \left| \hat f \hat g (\omega) \right|
\le B_2 $ for almost all $\omega \in E.$ Hence

$$\frac{A_2}{B_1} \le \left| \hat g (\omega) \right|
\le \frac{B_2}{A_1} \text{ for almost all } \omega \in E.$$

\medskip

The proofs of  \ref{fs2}. and \ref{b2}. are analogous.

\end{proof}

The following proposition gives necessary and sufficient conditions in order that the  union of frame-sequences  of irregular translations
 is  a frame sequence. The sufficient condition was proved first in \cite{ACM04}. We include a proof here for completeness.

\begin{proposition}
Let $\{E_j\}_{j\in J}$ be a family of subsets of bounded subsets of $\RRd$ such
that $h_j \in P_{E_j}$ for all $j\in J.$ Assume that
$\{e_{\lambda_k}\chi_{E_j}\}_{k \in K}$ is a frame for
$L^2(E_j)$ with frame bounds $m_j$ and $M_j$ for every $j\in J.$ If
$m=\inf_j m_j > 0$ and $M=\sup_j M_j < +\infty,$ then

$\{T_{\lambda_k}h_j\}_{k \in K, \, j \in J}$ is a frame
for $P_{\bigcup_{j\in J}E_j}$ if and only if there exist constants
$0<p\leq P$ such that $$p\leq \sum_{j\in J}|\hat{h_j}(w)|^2 \leq
P\,\,\,\,\,\,\,\,\, a.e. \, in \,\,\,\,\bigcup_{j\in J}E_j.$$
\end{proposition}

\begin{proof}

$\Longleftarrow$)

Let $f\in P_{\bigcup_{j\in J}E_j}.$

\begin{equation*}
\begin{split}
\sum_{k \in K, \, j \in J} &\left|\langle T_{\lambda_k}h_j,
f \rangle_{L^2(\RRd)}\right|^2 \\
& =\sum_{k \in K, \, j \in
J}\left|\langle e_{\lambda_k}\chi_{E_j}\hat{h_j},\hat{f}
\rangle_{L^2(\RRd)}\right|^2 = \sum_{k \in K, \, j \in J}\left|\langle
e_{\lambda_k}\chi_{E_j}\hat{h_j},\hat{f} \rangle_{L^2(E_j)}\right|^2\\
& =
\sum_{j \in J}\sum_{k \in K}\left|\langle
e_{\lambda_k}\chi_{E_j},\overline{\hat{h_j}}\hat{f}
\rangle_{L^2(E_j)}\right|^2\leq \sum_{j \in
J}M_j\|\overline{\hat{h_j}}\hat{f}\|^2\\
& \leq M  \sum_{j \in J}\int_{\RRd}|\overline{\hat{h_j}}\hat{f}|^2 (\omega) d \omega =
M\int_{\RRd}\sum_{j \in J}|\hat{h_j}(\omega) |^2  |\hat{f}|^2 (\omega) d \omega  \\
& \leq
MP\int_{\RRd}|\hat{f}|^2 (\omega) d \omega =MP\|f\|^2.
\end{split}
\end{equation*}

The other inequality can be proved analogously.

\bigskip

$\Longrightarrow)$

Let $\{T_{\lambda_k}h_j\}_{k \in K, \, j \in J}$ be a frame of
$P_{\bigcup_{j\in J}E_j}$ and assume that for every $p>0$ there
exists a a set $U \subseteq \bigcup_{j\in J}E_j$ of positive
measure such that $\sum_{j\in J}|\hat{h_j}(w)|^2 < p$ for every $w
\in U.$

For $n\in \NN$ define $E_n= \{w\in \bigcup_{j\in J}E_j: \sum_{j\in
J}|\hat{h_j}(w)|^2< \frac{1}{n}\}.$ Let $C_k=\{x\in \RRd:
k-1\leq\|x\|<k\}.$ We can write $E_n=\bigcup_{k\in\NN}E_n\cap C_k.$
Since $|E_n|>0$, there exists a $k_0\in
\NN$ such that $|E_n\cap C_{k_0}|>0.$ Let $A_n=E_n\cap
C_{k_0}.$ Since $0<|A_n|<+\infty$ we define

\begin{equation}
f_n(t)=
\begin{cases}
\frac{1}{\sqrt{|A_n|}}& \text{for } t \in A_n \\
 0 & \text{otherwise}.
\end{cases}
\end{equation}

Clearly $f_n \in L^2(\bigcup_{j\in J}E_j)$ for every $n\in\NN.$

If $\alpha$ is the lower frame bound of $\{T_{\lambda_k}h_j\}_{\lambda_k
\in \Lambda, \, j \in J}$ and $M$ is the upper frame bound of
$\{e_{\lambda_k}\chi_{E_j}\}_{k \in K},$ then

\begin{align*}
\alpha& \leq \sum_{k \in K, \, j \in J}\left|\langle T_{\lambda_k}h_j,
f_n\rangle\right|^2
 =\sum_{k \in K, \, j \in
J}\left|\langle e_{\lambda_k}\chi_{E_j}\hat{h_j},\hat{f_n} \rangle
\right|^2\\
&= \sum_{k \in K, \, j \in J}\left|\langle
e_{\lambda_k}\chi_{E_j},\overline{\hat{h_j}}\hat{f_n} \rangle
\right|^2
\leq \sum_{j\in J}M_j \|f_n\hat{h}\|^2\\
& \leq M\sum_{j\in
J}\int_{A_n}|f_n\hat{h_j}|^2 (\omega) d \omega = \frac{1}{|A_n|}M\int_{A_n}\sum_{j\in
J}|\hat{h_j}|^2 (\omega) d \omega \\
& \leq \frac{M}{n}\rightarrow 0,
\end{align*}

which is a contradiction. The other frame inequality can be proved
analogously.

\end{proof}

\section*{Acknowledgements}
P. Balazs was partially supported by the WWTF project MULAC
('Frame Multipliers: Theory and Application in Acoustics;
MA07-025). The research of  C. Cabrelli and U. Molter is partially
supported by Grants UBACyT X149 and X028 (UBA), PICT 2006-00177
(ANPCyT), and PIP 112-200801-00398 (CONICET) and S. Heineken
acknowledges the support of the Intra-European Marie Curie
Fellowship (FP7 project PIEF-GA-2008-221090 carried out at NuHAG,
Faculty of Mathematics, University of Vienna).

The seeds of this work were developed during a visit of C.~Cabrelli, S.~Heineken and U.~Molter to the NuHAG during the special semester at the ESI (2005) in Vienna. We thank NuHAG and specially Hans Feichtinger for the hospitality during the visit.



%
%
%
%
%
%

\end{document}